\documentclass{article}

\usepackage[utf8]{inputenc}
\usepackage{hyperref,xcolor}
\usepackage{caption,subcaption,longtable}
\usepackage{graphics, setspace}
\usepackage{array,multirow}
\usepackage{tikz,float}
\usepackage{amsmath,amsthm,amssymb}
\usepackage{enumerate,enumitem,listings}
\usepackage[T1]{fontenc}

\theoremstyle{plain}
\newtheorem{theoremalph}{Theorem}

\newtheorem{lemma}{Lemma}

\newtheorem*{theorem*}{Theorem}
\newtheorem*{lemma*}{Lemma}

\theoremstyle{definition}
\newtheorem{definition}[lemma]{Definition}
\newtheorem{remark}[lemma]{Remark}

\setlist[enumerate]{align=left}
\title{On the eigenvalues of Erd\"os--R\'enyi random bipartite graphs}
\author{Calum J. Ashcroft}
\date{\vspace{-2em}}

\begin{document}
\maketitle
	\begin{abstract}
We analyse the eigenvalues of Erd\"os--R\'enyi random bipartite graphs. In particular, we consider $p$ satisfying $n_{1}p=\Omega(\sqrt{n_{1}p}\log^{3}(n_{1})),$ $n_{2}p=\Omega(\sqrt{n_{2}p}\log^{3}(n_{2})),$ and let $G\sim G(n_{1},n_{2},p)$. We show that with probability tending to $1$ as $n_{1}$ tends to infinity: $$\mu_{2} (A(G))\leq 2[1+o(1)](\sqrt{n_{1}p}+\sqrt{n_{2}p}+\sqrt{(n_{1}+n_{2})p}).$$
	\end{abstract}


\section{Introduction}

\subsection{Statement of results}
It is well known that many models of random graphs are expanders. In particular, if $A$ is the adjacency matrix of a graph $G$ on $n$ vertices, we define the ordering of the eigenvalues of $A$ by $\mu_{1}(A)\geq \mu_{2}(A)\geq \hdots\geq \mu_{n}(A)$ (we keep this ordering convention for the eigenvalues of any symmetric matrix). We further define $\mu (A)=\max\{\vert\mu_{2}(A)\vert,\vert\mu_{n}(A)\vert\}.$ We are concerned with how large $\mu (A)$ can be, relative to the average degree of $G$. A set of strongly related quantities are the eigenvalues $\mu_{i}(G):=\mu_{i}(I-D^{-1\slash 2}AD^{-1\slash 2})$, where $D$ is the degree matrix of $G$.

The Alon--Boppana bound states that for a $d$-regular graph, $\mu (A)\geq 2\sqrt{d-1}\\-o(1)$ \cite{Alon}. A major result due to Friedman is that random $d$-regular graphs are almost Ramanujan, i.e. for any $\epsilon>0$, $\mu (A)\leq 2\sqrt{d-1}+\epsilon$ with probability tending to $1$ as $n$ tends to infinity \cite{Friedman}, while Bordenave proved that $\mu (A)\leq 2\sqrt{d-1}+o(1)$ \cite{bordenave2015new}.

We note that for $p$ sufficiently large, and $G\sim G(n,p)$ the Erd\"os--R\'enyi random graph, $G$ is almost $np$-regular with probability tending to $1$ as $n$ tends to infinity. Furthermore, such a graph also satisfies $\mu (A)\leq 2[1+o(1)]\sqrt{np}$ \cite{furedikomlos}: these results were then extended to a more general model of random graphs with given expected degree sequence by \cite{chungrandomgraph}.

Switching to random bipartite graphs, since the eigenvalues of bipartite graphs are symmetric around zero, we need only consider $\mu_{2}(A)$. The analogue of the Alon--Boppana bound for a $(d_{L},d_{R})$-regular bipartite graph is\\
$\mu_{2} (A)\geq \sqrt{d_{L}-1}+\sqrt{d_{R}-1}-\epsilon$ for $\epsilon>0$ and the number of vertices sufficiently large \cite{Feng_Li,Li_Sole}. Furthermore, the bound is almost attained for random $(d_{L},d_{R})$-regular graphs: with probability tending to $1$ as the number, $n$, of vertices tends to infinity, for sequences $\epsilon_{n},\epsilon_{n}'$ and for $G$ a random $(d_{L},d_{R})$-regular graph, $\mu_{2}(A)\leq \sqrt{d_{L}-1}+\sqrt{d_{R}-1}+\epsilon_{n}$, and $$\mu_{+}(A)=\min_{i} \{\mu_{i}(A)>0\}\geq \sqrt{d_{L}-1}+\sqrt{d_{R}-1}-\epsilon'_{n}\; \cite{brito2018spectral}.$$

There are many other results concerning the eigenvalues of random bipartite graphs: both \cite{Dumitriu_Johnson} and \cite{tranrandombipartite} study the spectral distribution of random biregular bipartite graphs, and show it converges to certain laws when $\vert V_{1}\vert \slash \vert V_{2}\vert$ tends to a limit $\alpha\neq 0,\infty$, though each of the two considers a different range of $(d_{L},d_{R})$. 
Finally, the second eigenvalue of the matrix $D-A$ is considered for certain random biregular bipartite graphs in \cite{zhu2020second}.

For $n_{2}=n_{2}(n_{1})\geq n_{1}$, and $0\leq p=p(n_{1})\leq 1$, we define the Erd\"os--R\'enyi random bipartite graph $G(n_{1},n_{2},p)$ as the graph with vertex partition $V_{1}=\{u_{1},\hdots ,u_{n_{1}}\}$ and $V_{2}=\{v_{1},\hdots ,v_{n_{2}}\}$, and edge set obtained by adding each edge $(u_{i},v_{j})$ independently with probability $p$. We write $G\sim G(n_{1},n_{2},p)$ to indicate that the graph $G$ is obtained by this process.

To our knowledge, there are no results in the literature on the eigenvalues of $G(n_{1},n_{2},p)$: while we do not show such graphs almost attain the Alon-Boppana bound, we are still able to prove that they are within a multiplicative constant of this bound, and so $\mu_{i}(G)$ is close to $1$ for any $i\neq 1,n_{1}+n_{2}.$ 
\begin{theoremalph}\label{mainthm: eigenvalue of random bipartite graphs}
Let $n_{1}\geq 1$, $n_{2}=n_{2}(n_{1})$, and $p=p(n_{1})$ be such that $$n_{1}p=\Omega(\sqrt{n_{1}p}\log^{3}(n_{1})),\mbox{ and }n_{2}p=\Omega(\sqrt{n_{2}p}\log^{3}(n_{2})).$$ Let $G\sim G(n_{1},n_{2},p).$ Then with probability tending to $1$ as $n_{1}$ tends to infinity: $$\mu_{2} (A(G))\leq 2[1+o(1)]\bigg(\sqrt{(n_{1}+n_{2})p}+\sqrt{n_{1}p}+\sqrt{n_{2}p}\bigg),$$
and so with probability tending to $1$ as $n_{1}$ tends to infinity:
$$\max\limits_{i\neq 1,n_{1}+n_{2}}\bigg\vert\mu_{i}(G)-1\bigg\vert\leq 2[1+o(1)]\bigg(\sqrt{\frac{1}{n_{1}p}+\frac{1}{n_{2}p}}+\frac{1}{\sqrt{n_{1}p}}+\frac{1}{\sqrt{n_{2}p}}\bigg).$$
\end{theoremalph}

\subsection{Notation and definitions}
We now briefly discuss some notation and definitions. We sometimes arrive at situations where $n$ is some parameter tending to infinity that is required to be an integer: if $n$ is not integer, we will implicitly replace it by $\lfloor n\rfloor$. Since we are dealing with asymptotics, this does not affect any of our arguments. We now define the following notation.
\begin{definition}
   Let $f,g:\mathbb{N}\rightarrow \mathbb{R}_{+}$ be two functions. We write $f=o(g)$ if $f(n)\slash g(n)\rightarrow 0$ as $n\rightarrow\infty$, and $f=\Omega (g)$ if $g=o(f)$.
\end{definition}

Let $G=(V,E)$ be a graph with vertex set $V=\{x_{1},\hdots ,x_{n}\}$. Note that the vertex partition of a bipartite graph $G$ will always be written $V(G)=V_{1}(G)\sqcup V_{2}(G)$.

The \emph{adjacency matrix} of $G$, $A(G)$, is the $n\times n$ matrix with $A(G)_{i,j}$ defined to be the number of edges between $x_{i}$ and $x_{j}$. The \emph{degree matrix} of $G$, $D(G)$, is the diagonal matrix with entries $D(G)_{i,i}=deg(x_{i})$. 

We note the following lemma, commonly known as Weyl's inequality, which will be of great use. 

\begin{lemma*}[Weyl's inequality]
Let $A$ and $B$ be symmetric $n\times n$ real matrices. For $i=1,\hdots ,n$: $\mu_{i}(A)+\mu_{n}(B)\leq \mu_{i}(A+B)\leq \mu_{i}(A)+\mu_{1}(B).$
\end{lemma*}

We also make the following remark. 
\begin{remark}
Let $M$ be a symmetric $n\times n$ matrix. For $i=1,\hdots , n$: $$\mu_{i}(-M)=-\mu_{n+1-i}(A).$$ This follows as $\{\mu_{i}(-M)\;:\;1\leq i\leq n\}=\{-\mu_{i}(M)\;:\;1\leq i\leq n\}$
and $\mu_{1}(A)\geq \mu_{2}(A)\geq \hdots\geq \mu_{m}(A),$ so that $-\mu_{1}(A)\leq -\mu_{2}(A)\leq \hdots\leq -\mu_{m}(A).$
\end{remark}
\section*{Acknowledgements}
As always, I would like to thank my supervisor Henry Wilton for his guidance and support. 
\section{The spectra of Erd\"os--R\'enyi random bipartite graphs}\label{sec: spectral theory of random bipartite graphs}
In this section we now analyse the spectra of Erd\"os--R\'enyi random bipartite graphs. We use the following result from \cite{furedikomlos} (c.f. \cite{chungrandomgraph}).

\begin{lemma}\label{lem: eigenvalue of ER random graph}
Let $n\geq 1$, let $p$ be such that $np=\Omega(\sqrt{np}\log^{3}(n))$, and let $G\sim G(n,p)$. Then with probability tending to $1$ as $n$ tends to infinity,
$\mu(A(G)) = 2[1+o(1)] \sqrt{np}$.
\end{lemma}
We now prove our main theorem.

\begin{proof}[Proof of Theorem \ref{mainthm: eigenvalue of random bipartite graphs}]
Let $G\sim G(n_{1},n_{2},p)$, let $A=A(G)$, and $D=D(G)$. Let $G'$ be the graph obtained by adding to $G$ each (non-loop) edge in $V_{1}^{2}$ with probability $p$ and each (non-loop) edge in $V_{2}^{2}$ with probability $p$. Then $G'\sim G(n_{1}+n_{2},p)$. By assumption, $n_{1}p=\Omega (\sqrt{n_{1}p}\log^{3}n_{1})$ and $n_{2}p=\Omega (\sqrt{n_{2}p}\log^{3}n_{2})$, so that $(n_{1}+n_{2})p=\Omega (\sqrt{(n_{1}+n_{2})p}\log^{3}(n_{1}+n_{2})).$
The adjacency matrix of $G'$ is of the form 
$$A(G')=\begin{pmatrix} 
 A_{1}&A_{2}\\
 A_{2}^{T}&A_{3}
\end{pmatrix},$$
where $A_{1}$ is the adjacency matrix of a $G(n_{1},p)$ graph, $A_{3}$ is the adjacency matrix of a $G(n_{2},p)$ graph,
and the matrix 
$$\begin{pmatrix}
 0&A_{2}\\
 A_{2}^{T}&0
\end{pmatrix}$$
is the adjacency matrix of $G$. By Lemma \ref{lem: eigenvalue of ER random graph}, with probability tending to $1$ as $n_{1}$ tends to infinity, $\max\limits_{i\neq 1}\vert \mu_{i}(A(G'))\vert=2[1+o(1)] \sqrt{(n_{1}+n_{2})p}.$
Let
$$A_{1}'=\begin{pmatrix} A_{1}&0\\
                           0&0
            \end{pmatrix},\; A=A(G(n_{1},n_{2},p))=\begin{pmatrix} 0&A_{2}\\
                            A_{2}^{T}&0
            \end{pmatrix},\; A_{3}'=\begin{pmatrix} 0&0\\
                            0&A_{3}
            \end{pmatrix}.$$
By the assumptions on $p$, we see that $p$ satisfies the requirements of Lemma \ref{lem: eigenvalue of ER random graph} for $G(n_{1},p)$ and $G(n_{2},p)$, so that with probability tending to $1$ as $n_{1}$ (and hence $n_{2}\geq n_{1}$) tends to infinity: $$\vert\max_{i\neq 1}\mu_{i}(A_{1})\vert =2[1+o(1)]\sqrt{n_{1}p},\;\vert\max_{i\neq 1}\mu_{i}(A_{3})\vert =2[1+o(1)]\sqrt{n_{2}p};$$ this clearly holds for $A_{1}', A_{3}'$ also.
Hence $$\mu_{1}(-A_{1}')=-\mu_{n_{1}+n_{2}}(A_{1}')\leq 2 [1+o(1)] \sqrt{n_{1}p},$$ and similarly $\mu_{1}(-A_{3}')\leq 2[1+o(1)]\sqrt{n_{2}p}.$
Therefore, by Weyl's inequality, with probability tending to $1$ as $n_{1}$ tends to infinity:
\begin{equation*}
\begin{split}
   \mu_{2}(A) &=  \mu_{2}\bigg(A(G')- A_{1}'-A_{3}'\bigg)\leq \mu_{2}(A(G'))+\mu_{1}(-A_{1}')+\mu_{1}(-A_{3}')\\
    &\leq 2[1+o(1)]\bigg(\sqrt{(n_{1}+n_{2})p}+\sqrt{n_{1}p}+\sqrt{n_{2}p}\bigg).
\end{split}
\end{equation*}
Furthermore, by a routine application of Chebyshev's inequality, $G$ is almost $(n_{2}p,n_{1}p)$-regular with probability tending to $1$ as $n_{1}$ tends to infinity,  i.e. the minimum and maximum degree of vertices in $V_{1}(G)$ are $(1+o(1))n_{2}p$ and the minimum and maximum degree of vertices in $V_{2}(G)$ are $(1+o(1))n_{1}p$. Therefore, we see that there exists a matrix $K$ with norm $\vert\vert K\vert\vert_{\infty}=o\bigg(1\bigg\slash \sqrt{n_{1}n_{2}p^{2}}\bigg)$ such that $$\frac{1}{\sqrt{n_{1}n_{2}p^{2}}}A=D^{-1\slash 2}AD^{-1\slash 2}+K.$$
Since $\max_{j}\vert \mu_{j}(K)\vert \leq \vert \vert K\vert \vert_{\infty}$, we see that for $i=1,\hdots ,n_{1}+n_{2}$: $$\mu_{i}\bigg(\frac{1}{\sqrt{n_{1}n_{2}p^{2}}}A\bigg)=\mu_{i}\bigg(D^{-1\slash 2}AD^{-1\slash 2}\bigg)+o\bigg(\frac{1}{\sqrt{n_{1}n_{2}p^{2}}}\bigg).$$
As $1-\mu_ {i}(G)=\mu_{n_{1}+n_{2}-i+1}(D^{-1\slash 2}AD^{-1\slash 2})$ for $i=1,\hdots ,n_{1}+n_{2}$, the result follows.
\end{proof}

	\bibliographystyle{alpha}
	\bibliography{bib}
	{\sc{DPMMS, Centre for Mathematical Sciences, Wilberforce Road, Cambridge, CB3 0WB, UK}}\\
	\emph{E-mail address}: cja59@dpmms.cam.ac.uk
\end{document}